\documentclass[12pt,twoside,reqno]{amsart}

\usepackage{microtype}
\usepackage[OT1]{fontenc}
\usepackage{type1cm}
\usepackage{amssymb}
\usepackage[left=2.3cm,top=3cm,right=2.3cm]{geometry}

\usepackage{graphics,color}
\usepackage{color}

\numberwithin{equation}{section}

\theoremstyle{plain}
\newtheorem{theorem}{Theorem}[section]
\newtheorem{maintheorem}{Theorem}

\newtheorem{proposition}[theorem]{Proposition}
\newtheorem{lemma}[theorem]{Lemma}

\theoremstyle{remark}
\newtheorem{remark}[theorem]{Remark}

\newtheorem{example}[theorem]{Example}
\newtheorem*{ack}{Acknowledgement}

\theoremstyle{definition}

\newtheorem{fact}{Fact}

\newcommand*{\myproofname}{Proof}
\newenvironment{myproof}[1][\myproofname]{\begin{proof}[#1]}{\end{proof}}

\newcommand{\HH}{\mathcal{H}}

\newcommand{\R}{\mathbb{R}}

\newcommand{\N}{\mathbb{N}}

\newcommand{\eps}{\varepsilon}
\newcommand{\fii}{\varphi}

\DeclareMathOperator{\diam}{diam}

\DeclareMathOperator{\diag}{diag}
\DeclareMathOperator{\image}{Img}

\begin{document}

\title{Self-affine sets in analytic curves and algebraic surfaces}

\author{De-Jun Feng}
\address{Department of Mathematics \\
         The Chinese University of Hong Kong \\
         Shatin \\
         Hong Kong}
\email{djfeng@math.cuhk.edu.hk}

\author{Antti K\"aenm\"aki}
\address{Department of Mathematics and Statistics \\
         P.O. Box 35 (MaD) \\
         FI-40014 University of Jyv\"askyl\"a \\
         Finland}
\email{antti.kaenmaki@jyu.fi}

\date{\today}

\begin{abstract}
  We characterize analytic curves that contain non-trivial self-affine sets. We also prove that compact algebraic surfaces do not contain non-trivial self-affine sets.
\end{abstract}

\maketitle

\section{Introduction}

Self-similar and self-affine sets  are among the most typical and important fractal objects; see e.g.\ \cite{Falconer1990}.  They can be generated by the so-called iterated function systems; see Section \ref{sec:preli}. Although these sets can be very irregular as one expects, they often have very rigid geometric structure.

It is not surprising that typical non-flat smooth manifolds do not contain any non-trivial self-similar or self-affine set. For instance, circles are such examples. To see this, suppose to the contrary that a circle $C$  contains a non-trivial self-affine set $E$. Let $f$ be a contractive  affine map in the defining iterated function system of $E$. Then $f(E) \subset E$ and thus $f(E)$ is contained in both $C$ and $f(C)$. However, since $f(C)$ is an ellipse with diameter strictly smaller than that of $C$, the intersection of $f(C)$ and $C$ contains at most two points. This is a contradiction since $f(E)$ is an infinite set.

The above general phenomena was first clarified  by Mattila~\cite{Mattila1982} in the self-similar case.  He proved that a self-similar set $E$ satisfying the open set condition either lies on an $m$-dimensional affine subspace or $\HH^t(E \cap M) = 0$ for every $m$-dimensional $C^1$-submanifold of $\R^n$. Here $t$ is the Hausdorff dimension of $E$ and $\HH^t$ is the $t$-dimensional Hausdorff measure. This  result  was later generalized to self-conformal sets in \cite{Kaenmaki2003, Kaenmaki2006, MayerUrbanski2003}. As a related work,  Bandt and Kravchenko~\cite{BandtKravchenko2011} showed that if $E$ is a self-similar set which spans $\R^n$ and $x \in E$, then there does not exist a tangent hyperplane of $E$ at $x$.

As an easy consequence of the result of Mattila or that of Bandt and Kravchenko, an analytic planar curve does not contain any non-trivial self-similar set unless it is a straight line segment. In a private communication, Mattila asked  which kind of analytic planar curves  can contain a non-trivial self-affine set. The main purpose of this article is to answer this question.

We first remark that any closed parabolic arc is a self-affine set. This interesting fact was first pointed out by Bandt and Kravchenko~\cite{BandtKravchenko2011}. In that paper, they considered self-affine planar curves consisting of two pieces $E = f_1(E) \cup f_2(E)$. They showed that if a certain condition on the eigenvalues of $f_1$ and $f_2$ holds, then the curve $E$ is differentiable at all except for countably many points. They also introduced a stronger condition on the eigenvalues which guarantees the curve $E$ to be continuously differentiable. This result implies that there exist many continuously differentiable self-affine curves. However, Bandt and Kravchenko furthermore showed that self-affine curves cannot be very smooth: the only simple $C^2$ self-affine planar curves are parabolic arcs and straight lines.

In our main result, instead of curves that are itself self-affine, we consider general self-affine sets and examine when they can be contained in an analytic curve.

\begin{maintheorem} \label{thm:main}
 An analytic curve in $\R^n$, $n \ge 2$, which cannot be embedded in a hyperplane contains a non-trivial self-affine set if and only if it is an affine image of $\eta \colon [c,d] \to \R^n$, $\eta(t) = (t,t^2,\ldots,t^n)$, for some $c<d$.
\end{maintheorem}

The above result gives a complete answer to the question of Mattila: the only analytic planar curves that contain non-trivial self-affine sets are parabolic arcs and straight line segments. As explained by Mattila, the question is related to the study of singular integrals and self-similar sets in Heisenberg groups. In such groups, self-similar sets are self-affine in the Euclidean metric. From the singular integral theory point of view, it is thus important to understand when a self-affine set is contained in an analytic manifold.

Concerning manifolds, we study an analogue of Mattila's question. We examine which kind of algebraic surfaces can contain self-affine sets. Our result shows that this cannot happen on compact surfaces.

\begin{maintheorem} \label{thm:no_self-affine}
  A compact algebraic surface does not contain non-trivial self-affine sets.
\end{maintheorem}

It is easy to see that non-compact surfaces, such as paraboloids, can contain non-trivial self-affine sets; see Example \ref{ex:paraboloid}. To finish the article, we introduce in Proposition \ref{thm:selfaffinepoly} a sufficient condition for the inclusion of a self-affine set in an algebraic surface.

\section{Preliminaries} \label{sec:preli}

In this section, we introduce the basic concepts to be used throughout in the article. A mapping $f \colon \R^n \to \R^n$ is \emph{affine} if $f(x) = Tx + c$ for all $x \in \R^n$, where $T$ is a $n \times n$ matrix and $c \in \R^n$. The matrix $T$ is called a \emph{linear part} of $f$. It is easy to see that an affine map is invertible if and only if its linear part is non-singular. A mapping $f \colon \R^n \to \R^n$ is \emph{strictly contractive} if $|f(x)-f(y)| < |x-y|$ for all $x,y \in \R^n$. Note that an affine mapping $f$ is strictly contractive if and only if its linear part $T$ has operator norm $\| T \|$ strictly less than $1$. A non-empty compact set $E \subset \R^n$ is called \emph{self-affine} if $E = \bigcup_{i=1}^\ell f_i(E)$, where $\{ f_i \}_{i=1}^\ell$ is an \emph{affine iterated function system (IFS)}, i.e.\ a finite collection of strictly contractive invertible affine maps $f_i \colon \R^n \to \R^n$; see \cite{Hutchinson1981}. Moreover, $E$ is called \emph{self-similar} if all the $f_i$'s are similitudes. We say that a self-affine set is \emph{non-trivial} if it is not a singleton.

If $a<b$, then a non-constant continuous function $\gamma \colon [a,b] \to \R^n$ is called a \emph{curve}. We denote the set $\gamma([a,b]) \subset \R^n$ by $\image(\gamma)$ and refer to it also as a \emph{curve}. By saying that a curve $\gamma$ contains a set $A$ we obviously mean that $A \subset \image(\gamma)$. A curve $\gamma$ is \emph{simple} if $\gamma(s) \ne \gamma(t)$ for $a \le s<t<b$. We say that a curve $\gamma \colon [a,b] \to \R^n$, $\gamma(t) = (x_1(t),\ldots,x_n(t))$, is \emph{analytic} if $x_i \colon [a,b] \to \R$ is continuous on $[a,b]$ and real analytic on $(a,b)$ for all $i \in \{ 1,\ldots,n \}$. Recall that a function is real analytic on an open set $U \subset \R$ if, at any point $t \in U$, it can be represented by a convergent power series on some interval of positive radius centered at $t$. Similarly, if $x_i$'s are $C^k$ functions for some $k \in \N$, then the curve $\gamma$ is called \emph{$C^k$ curve}. The $k$-th derivative of a $C^k$ curve $\gamma$ is $\gamma^{(k)}(t) = (x_1^{(k)}(t),\ldots,x_n^{(k)}(t))$.
If $f \colon \R^n \to \R^n$ is an invertible affine mapping and $\gamma \colon [a,b] \to \R^n$ is a curve, then $f \circ \gamma$ is the \emph{affine image} of the curve.

Let $P \colon \R^n \to \R$ be a non-constant polynomial with real coefficients. The set
\begin{equation*}
  S(P) = \{ x \in \R^n : P(x)=0 \}
\end{equation*}
is called an \emph{algebraic surface}. The \emph{degree} of $P$, denoted by $\deg(P)$, is the highest degree of its terms, when $P$ is expressed in canonical form. The degree of a term is the sum of the exponents of the variables that appear in it.

\section{Self-affine sets and analytic curves}

In this section, we prove Theorem \ref{thm:main}. Our arguments are inspired by the proof of \cite[Theorem 3(i)]{BandtKravchenko2011}. We will first show that an affine image of $\eta \colon [c,d] \to \R^n$, $\eta(t) = (t,t^2,\ldots,t^n)$, contains a non-trivial self-affine set. This follows immediately from the following lemma.

\begin{lemma} \label{thm:parabola-self-affine}
  If $\eta \colon [c,d] \to \R^n$, $\eta(t) = (t,t^2,\ldots,t^n)$, then $\image(\eta)$ is a non-trivial self-affine set for all $c<d$.
\end{lemma}

\begin{proof}
  Let
  \begin{equation*}
    0 < \lambda < (2^n\sqrt{n}\max\{ (2|c|+1)^n,(|c|+|d|+1)^n \})^{-1} < 1
  \end{equation*}
  and choose $t_1,\ldots,t_\ell \in [c,d]$ with $\ell \in \N$ such that the self-similar set of $\{ x \mapsto \lambda(x-c) + t_i \}_{i=1}^\ell$ is $[c,d]$. Write $c_{i,k,j} = \binom{k}{j} (\frac{t_i}{\lambda}-c)^{k-j}$ and observe that
  \begin{equation*}
    \Bigl( t - \Bigl( c-\frac{t_i}{\lambda} \Bigr) \Bigr)^k = \sum_{j=1}^k c_{i,k,j} \Bigl( t^j - \Bigl( c-\frac{t_i}{\lambda} \Bigr)^j \Bigr)
  \end{equation*}
  for all $k \in \{ 1,\ldots,n \}$, $i \in \{ 1,\ldots,\ell \}$, and $t \in \R$.

  Defining for each $i \in \{ 1,\ldots,\ell \}$ a lower-triangular matrix by
  \begin{equation*}
    T_i =
    \begin{pmatrix}
      \lambda c_{i,1,1}   & 0                   & 0                   & \cdots & 0 \\
      \lambda^2 c_{i,2,1} & \lambda^2 c_{i,2,2} & 0                   & \cdots & 0 \\
      \lambda^3 c_{i,3,1} & \lambda^3 c_{i,3,2} & \lambda^3 c_{i,3,3} & \cdots & 0 \\
      \vdots              & \vdots              & \vdots              & \ddots & \vdots \\
      \lambda^n c_{i,n,1} & \lambda^n c_{i,n,2} & \lambda^n c_{i,n,3} & \cdots & \lambda^n c_{i,n,n}
    \end{pmatrix},
  \end{equation*}
  we see, by the choice of $\lambda$ and the fact that $t_i \in [c,d]$, that
  \begin{align*}
    \| T_i \| &\le \sqrt{n} \max_{k \in \{ 1,\ldots,n \}} \sum_{j=1}^k |\lambda^kc_{i,k,j}| = \sqrt{n} \max_{k \in \{ 1,\ldots,n \}} \sum_{j=1}^k \lambda^k \binom{k}{j} \Bigl| \frac{t_i}{\lambda}-c \Bigr|^{k-j} \\
    &\le \sqrt{n} \max_{k \in \{ 1,\ldots,n \}} \sum_{j=1}^k \lambda^j \binom{k}{j} (|t_i|+|c|+1)^{k-j} \le \lambda\sqrt{n} \max_{k \in \{ 1,\ldots,n \}} (|t_i|+|c|+1)^k 2^k < 1.
  \end{align*}
  Therefore, the affine map $f_i \colon \R^n \to \R^n$ defined by
  \begin{equation*}
    f_i(x_1,\ldots,x_n) = T_i(x_1,\ldots,x_n) - T_i\Bigl( c-\frac{t_i}{\lambda},\Bigl( c-\frac{t_i}{\lambda} \Bigr)^2,\ldots,\Bigl( c-\frac{t_i}{\lambda} \Bigr)^n \Bigr)
  \end{equation*}
  is contractive and satisfies
  \begin{align*}
    f_i(t,t^2,\ldots,t^n) &= T_i\Bigl( t - \Bigl( c-\frac{t_i}{\lambda} \Bigr), t^2 - \Bigl( c-\frac{t_i}{\lambda} \Bigr)^2, \ldots, t^n - \Bigl( c-\frac{t_i}{\lambda} \Bigr)^n \Bigr) \\
    &= \Bigl( \lambda\Bigl( t - \Bigl( c-\frac{t_i}{\lambda} \Bigr) \Bigr), \lambda^2\Bigl( t - \Bigl( c-\frac{t_i}{\lambda} \Bigr) \Bigr)^2, \ldots, \lambda^n\Bigl( t - \Bigl( c-\frac{t_i}{\lambda} \Bigr) \Bigr)^n \Bigr) \\
    &= (\lambda(t-c)+t_i, (\lambda(t-c)+t_i)^2, \ldots, (\lambda(t-c)+t_i)^n)
  \end{align*}
  for all $t \in [c,d]$. Hence the self-affine set of $\{ f_i \}_{i=1}^\ell$ is the curve $\image(\eta)$.
\end{proof}

\begin{remark}
  The key fact implicitly used in the above proof is that $\eta(t) = (t, t^2,\ldots, t^n)$ defined on $\R$ is invariant under homotheties $\diag(s,s^2,\ldots,s^n)$ and translations $(t,t^2,\ldots,t^n) \mapsto (t-a,(t-a)^2,\ldots,(t-a)^n)$.
\end{remark}

Let us next focus on the opposite claim.

\begin{theorem}
  If an analytic curve which cannot be embedded in a hyperplane contains a non-trivial self-affine set, then it is an affine image of $\eta \colon [c,d] \to \R^n$, $\eta(t) = (t,t^2,\ldots,t^n)$, for some $c<d$.
\end{theorem}

\begin{proof}
  Let $\gamma \colon [a,b] \to \R^n$ be an analytic curve such that $\image(\gamma)$ is not contained in a hyperplane. Suppose that $E$ is a non-trivial self-affine set of an affine IFS $\{ f_i \}_{i=1}^\ell$ such that $E \subset \image(\gamma)$. Let $\mathcal{S}$ be the semigroup generated by $f_1,\ldots,f_\ell$ under composition.

  By analyticity and the assumption that $\image(\gamma)$ is not contained in a hyperplane, without loss of generality, we may assume that $E \subset \gamma((a,b))$ and $\gamma'(t) \ne 0$ for all $t \in (a,b)$. Since $(a,b)$ has a countable cover of open intervals $I_i$ such that $\gamma(I_i)$ has no intersection points, we have $E \subset \bigcup_i E \cap \gamma(I_i)$ and therefore, by the Baire Category Theorem, there exist $i$ and an open set $U$ such that $\emptyset \ne E \cap U \subset E \cap \gamma(I_i)$. Since $E \cap U$ contains a non-trivial self-affine set, we see that no generality is lost if we assume the curve $\gamma$ to be simple.

  Fix $\fii \in \mathcal{S}$ and write
  \begin{equation} \label{eq:dj1}
    \fii(x) = M(x-x_0)+x_0
  \end{equation}
  for all $x \in \R^n$, where $x_0 \in \R^n$ is the fixed point of $\fii$ and $M$ is an $n \times n$ invertible matrix. Note that $x_0 \in E$. Since $E \subset \gamma((a,b))$ there exists $t_0 \in (a,b)$ such that $x_0 = \gamma(t_0)$. Hence we may rewrite \eqref{eq:dj1} as
  \begin{equation} \label{eq:dj2}
    \fii(x) = M(x-\gamma(t_0)) + \gamma(t_0).
  \end{equation}
  Since $E$ is non-trivial, there exists a sequence $(t_i)_{i \in \N}$ of distinct numbers in $(a,b)$ such that $t_i \to t_0$ as $i \to \infty$ and $\gamma(t_i) \in E$ for all $i \in \N$. Furthermore, since $\fii(E) \subset E \subset \gamma((a,b))$, we see that $\fii(\gamma(t_i)) \in \image(\gamma)$ and therefore, for each $i \in \N$ there exists $t_i' \in (a,b)$ such that
  \begin{equation} \label{eq:dj3}
    \fii(\gamma(t_i)) = \gamma(t_i').
  \end{equation}
  Recalling that $\gamma$ is simple and $\fii(\gamma(t_0)) = \gamma(t_0)$, we see that $t_i' \to t_0$ as $i \to \infty$. By \eqref{eq:dj2} and \eqref{eq:dj3}, we have
  \begin{equation} \label{eq:dj4}
    M(\gamma(t_i)-\gamma(t_0)) = \fii(\gamma(t_i))-\gamma(t_0) = \gamma(t_i')-\gamma(t_0)
  \end{equation}
  and therefore,
  \begin{equation*}
    M\biggl( \frac{\gamma(t_i)-\gamma(t_0)}{t_i-t_0} \biggr) = \frac{\gamma(t_i')-\gamma(t_0)}{t_i'-t_0} \cdot \frac{t_i'-t_0}{t_i-t_0}.
  \end{equation*}
  Letting $i \to \infty$, we have
  \begin{equation} \label{eq:lambda_1}
    M\gamma'(t_0) = \lambda\gamma'(t_0),
  \end{equation}
  where $\lambda = \lim_{i \to \infty} (t_i'-t_0)/(t_i-t_0) \ne 0$ by the invertibility of $M$.

  Let $J$ be an invertible matrix such that
  \begin{equation*}
    J^{-1}\gamma'(t_0) = (1,0,\ldots,0)
  \end{equation*}
  and
  \begin{equation*}
    J^{-1}MJ =
    \begin{pmatrix}
      A_1 & 0   & \cdots & 0 \\
      0   & A_2 & \cdots & 0 \\
      \vdots & \vdots & \ddots & \vdots \\
      0 & 0 & \cdots & A_m
    \end{pmatrix}
  \end{equation*}
  is a real canonical Jordan form of $M$. Write $A = J^{-1}MJ$ and recall that if $\lambda_i$ is a real eigenvalue of $M$, then
  \begin{equation*}
    A_i =
    \begin{pmatrix}
      \lambda_i & 1 & 0 & \cdots & 0 & 0 \\
      0 & \lambda_i & 1 & \cdots & 0 & 0 \\
      0 & 0 & \lambda_i & \cdots & 0 & 0 \\
      \vdots & \vdots & \vdots & \ddots & \vdots & \vdots \\
      0 & 0 & 0 & \cdots & \lambda_i & 1 \\
      0 & 0 & 0 & \cdots & 0 & \lambda_i
    \end{pmatrix},
  \end{equation*}
  and if $\lambda_i$ is a non-real eigenvalue of $M$ with real part $a_i$ and imaginary part $b_i$, then
  \begin{equation*}
    A_i =
    \begin{pmatrix}
      C_i & I & 0 & \cdots & 0 & 0 \\
      0 & C_i & I & \cdots & 0 & 0 \\
      0 & 0 & C_i & \cdots & 0 & 0 \\
      \vdots & \vdots & \vdots & \ddots & \vdots & \vdots \\
      0 & 0 & 0 & \cdots & C_i & I \\
      0 & 0 & 0 & \cdots & 0 & C_i
    \end{pmatrix},
  \end{equation*}
  where
  \begin{equation*}
    C_i =
    \begin{pmatrix}
      a_i & b_i \\ -b_i & a_i
    \end{pmatrix}
    \quad \text{and} \quad
    I =
    \begin{pmatrix}
      1 & 0 \\ 0 & 1
    \end{pmatrix}.
  \end{equation*}
  Note that by \eqref{eq:lambda_1}, we have $\lambda_1 = \lambda \in \R$. Moreover by \eqref{eq:dj4},
  \begin{equation} \label{eq:dj5}
    AJ^{-1}(\gamma(t_i)-\gamma(t_0)) = J^{-1}(\gamma(t_i')-\gamma(t_0))
  \end{equation}
  for all $i \in \N$.

  Defining $\tilde\gamma \colon [a,b] \to \R^n$ by
  \begin{equation}
  \label{e-new}
    \tilde\gamma(t) = J^{-1}(\gamma(t)-\gamma(t_0)),
  \end{equation}
  we clearly have $\tilde\gamma(t_0)=0$ and $\tilde\gamma'(t_0) = J^{-1}\gamma'(t_0) = (1,0,\ldots,0)$. Write $\tilde\gamma(t) = (\tilde x_1(t),\ldots,\tilde x_n(t))$. Then $\tilde x_1(t_0) =0$ and $\tilde x_1'(t_0) = 1 \ne 0$. By the inverse function theorem,  the function $\tilde x_1(t)$ has a local inverse $t=t(\tilde{x}_1)$ which is analytic on $(-\eps,\eps)$ for some $\eps>0$. Write $x_1^*=\tilde{x}_1$ and $x_k^*(x_1^*)=\tilde x_k(t(x_1^*))$ for $k\in \{ 2,\ldots,n \}$. Clearly  $ x_k^*(\cdot)$ is  analytic on $(-\eps,\eps)$ for all $k \in \{ 2,\ldots,n \}$. Note that
  \begin{equation} \label{e-new1}
    x_k^*(0) =\tilde x_k(t_0)= 0, \quad   (x_k^*)'(0)=\tilde x_k'(t_0)\cdot t'(0)=0
  \end{equation}
  for all $k \in \{ 2,\ldots,n \}$ and $x_2^*, \ldots,  x_n^*$ are not constant functions. Indeed, if $x_k^*$ was  constant for some $k$, then so is $\tilde x_k$; by the fact that each $\tilde x_k$ is a linear combination of $x_1, \ldots, x_n$ (see \eqref{e-new}), the curve $\gamma$ would be contained in a hyperplane in $\R^n$, leading to a contradiction.
  Let $\xi \colon (-\eps,\eps) \to \R^n$ be defined by
  \begin{equation} \label{eq:curve-eta}
    \xi( x_1^*) = (x_1^*, x_2^*( x_1^*),\ldots, x_n^*( x_1^*)).
  \end{equation}
  Then $\xi$ is a re-parametrization of the curve $\tilde\gamma$ restricted on a neighborhood of $t_0$. The goal of the proof is to show that an affine image of  the curve $\xi$ will be of the claimed form.

  Let us next collect three facts related to the above defined setting.

  \begin{fact} \label{fact1}
    Write $A = (a_{ij})_{1 \le i,j \le n}$ and let $Y = a_{11}x_1^*+\sum_{j=2}^n a_{1j}  x_j^*(x_1^*)$. Then
    \begin{equation} \label{eq:dj7}
      A(x_1^*, x_2^*(x_1^*),\ldots, x_n^*(x_1^*) )= (Y, x_2^*(Y),\ldots, x_n^*(Y))
    \end{equation}
    for all $x_1^* \in (-\eps,\eps)$.
  \end{fact}

  \begin{myproof}
    By \eqref{eq:dj5}, $A \tilde{\gamma}(t_i)= \tilde{\gamma}(t_i')$ for all $i\in \N$. Hence the equality \eqref{eq:dj7} holds for infinitely many different values of $x_1^*$ in a small closed neighborhood of $0$. By analyticity, \eqref{eq:dj7} holds on the whole interval $(-\eps,\eps)$.
  \end{myproof}

  The next fact concerns the shape of the matrix $A$.

  \begin{fact} \label{fact2}
    The matrix $A$ is diagonal. In other words, all the block matrices $A_i$ have dimension $1$.
  \end{fact}

  \begin{myproof}
    Let us first show that $A_1$ has dimension $1$. Suppose to the contrary that $d_1 = \dim(A_1)>1$. Since the eigenvalue associated to $A_1$ is $\lambda \in \R$, we have
    \begin{equation*}
      A_1 =
      \begin{pmatrix}
        \lambda & 1 & \cdots & 0 & 0 \\
        0 & \lambda & \cdots & 0 & 0 \\
        \vdots & \vdots & \ddots & \vdots & \vdots \\
        0 & 0 & \cdots & \lambda & 1 \\
        0 & 0 & \cdots & 0 & \lambda
      \end{pmatrix}.
    \end{equation*}
    By Fact \ref{fact1}, we see that
    \begin{equation} \label{eq:dj8}
      \lambda  x_{d_1}^*( x_1^*) =  x_{d_1}^*(\lambda  x_1^* +  x_2^*(x_1^*)).
    \end{equation}
    By \eqref{e-new1} and the fact that $x_k^*$, $k\in \{ 2,\ldots, n \}$, is not a constant, there exist integers $p_2,\ldots,p_n \ge 2$ and reals $c_2,\ldots,c_n \ne 0$ such that for each $k \in \{ 2,\ldots,n \}$
    \begin{equation} \label{eq:dj9}
      x_k^*(x_1^*) = c_k( x_1^*)^{p_k} + o( x_1^*)^{p_k}
    \end{equation}
    as $x_1^* \to 0$. Plugging \eqref{eq:dj9} into \eqref{eq:dj8}, and comparing the coefficients of Taylor series in $x_1^*$ on both sides, we get
    \begin{equation*}
      \lambda c_{d_1} = c_{d_1} \lambda^{p_{d_1}}
    \end{equation*}
    which implies that $p_{d_1} = 1$, a contradiction. Hence we have $\dim(A_1) = 1$ and therefore $Y = \lambda  x_1^*$.

    Let us next assume inductively that for some $k \in \{ 1,\ldots,n-1 \}$ the matrices $A_1,\ldots,A_k$ are of dimension $1$ and show that $\dim(A_{k+1})=1$. Suppose to the contrary that $d = \dim(A_{k+1}) > 1$. Now there are two cases: either $\lambda_{k+1}$ is real or not. First suppose that $\lambda_{k+1}$ is real. Let $\ell = k+d$. By \eqref{eq:dj7} we have
    \begin{align}
    \lambda_{k+1}x_{\ell-1}^*(x_1^*)+x_{\ell}^*( x_1^*)&=x_{\ell-1}^*(\lambda x_1^*), \label{e-l1}\\
    \lambda_{k+1}x_{\ell}^*(x_1^*)&=x_{\ell}^*(\lambda x_1^*). \label{e-l2}
    \end{align}
    Plugging \eqref{eq:dj9} into \eqref{e-l2}, and comparing the coefficients of Taylor series in $x_1^*$ on both sides, we get $\lambda_{k+1}=\lambda^{p_{\ell}}$. Then plug \eqref{eq:dj9} into \eqref{e-l1} to obtain
    \begin{equation*}
    \lambda^{p_{\ell}} c_{\ell-1} (x_1^*)^{p_{\ell-1}}+ c_{\ell} (x_1^*)^{p_{\ell}}=\lambda^{p_{\ell-1}} c_{\ell-1}(x_1^*)^{p_{\ell-1}}+o((x_1^*)^{p_{\ell-1}}+(x_1^*)^{p_{\ell}})
    \end{equation*}
    as $x_1^*\to 0$. That is, 
    \begin{equation}\label{e-l3}
    (\lambda^{p_{\ell}}-\lambda^{p_{\ell-1}}) c_{\ell-1} (x_1^*)^{p_{\ell-1}}+ c_{\ell} (x_1^*)^{p_{\ell}}=o((x_1^*)^{p_{\ell-1}}+(x_1^*)^{p_{\ell}})
    \end{equation}
    as $x_1^*\to 0$. Since $0<|\lambda|<1$ and $c_{\ell-1}, c_{\ell}\neq 0$,   one easily derives a contradiction from \eqref{e-l3} by considering the cases $p_{\ell}< p_{\ell-1}$, $p_{\ell}=p_{\ell-1}$, and $p_{\ell}>p_{\ell-1}$ separately.

  Hence we may assume that $\lambda_{k+1} = a+ib$ with $b \ne 0$. The matrix $A_{k+1}$ is therefore of the form
    \begin{equation*}
      A_{k+1} =
      \begin{pmatrix}
        a & b & 1 & 0 & \cdots & 0 & 0 \\
        -b & a & 0 & 1 & \cdots & 0 & 0 \\
        0 & 0 & a & b & \cdots & 0 & 0 \\
        0 & 0 & -b & a & \cdots & 0 & 0 \\
        \vdots & \vdots & \vdots & \vdots & \ddots & \vdots & \vdots \\
        0 & 0 & 0 & 0 & \cdots & a & b \\
        0 & 0 & 0 & 0 & \cdots & -b & a
      \end{pmatrix}.
    \end{equation*}
    Again let $\ell = k+d$. Applying \eqref{eq:dj7}, we see that
    \begin{align*}
      a x_{\ell-1}^*(x_1^*) + b x_\ell^*(x_1^*) &=  x_{\ell-1}^*(\lambda x_1^*), \\
      -bx_{\ell-1}^*(x_1^*) + a x_\ell^*(x_1^*) &=  x_\ell^*(\lambda x_1^*).
    \end{align*}
   Using the above identities and  comparing the coefficients of $(x_1^*)^{p_\ell}$ and $(x_1^*)^{p_{\ell-1}}$ in the Taylor expansions of $x_\ell^*$ and $x_{\ell-1}^*$, we see that $p_\ell = p_{\ell-1}$; and moreover,
    \begin{align*}
      ac_{\ell-1} + bc_\ell &= c_{\ell-1}\lambda^{p_\ell}, \\
      -bc_{\ell-1} + ac_\ell &= c_\ell \lambda^{p_\ell},
    \end{align*}
    or, equivalently,
    \begin{equation*}
      \begin{pmatrix}
        a & b \\ -b & a
      \end{pmatrix}
      \begin{pmatrix}
        c_{\ell-1} \\ c_\ell
      \end{pmatrix}
      = \lambda^{p_\ell}
      \begin{pmatrix}
        c_{\ell-1} \\ c_\ell
      \end{pmatrix}.
    \end{equation*}
    This means that the real number $\lambda^{p_\ell}$ is an eigenvalue of the above matrix, a contradiction.
  \end{myproof}

  By Fact \ref{fact2}, we may now write
  \begin{equation} \label{eq:A-diag}
    A = \diag(\lambda_1,\lambda_2,\ldots,\lambda_n),
  \end{equation}
  where $\lambda_1=\lambda \in (-1,1) \setminus \{ 0 \}$. With this observation, we can examine how the curve $\xi$ defined in \eqref{eq:curve-eta} looks like.

  \begin{fact} \label{fact3}
    There exist integers $2\leq p_2 < p_3 < \cdots < p_n$ such that a piece of the curve $\image(\gamma)$, namely $\gamma \colon (t_0-\delta_1,t_0+\delta_2) \to \R^n$ for some $\delta_1,\delta_2>0$, is an affine image of the curve $\eta \colon (-\eps,\eps) \to \R^n$ defined by
    \begin{equation*}
      \eta(t) = (t,t^{p_2},\ldots,t^{p_n}).
    \end{equation*}
    More precisely, there exists an invertible $n\times n$ matrix $B$ such that the above defined  $\eta$ is the re-parametrization of the curve $B(\gamma(t)-\gamma(t_0))$, $t\in (t_0-\delta_1, t_0+\delta_2)$.
  \end{fact}

  \begin{myproof}
  We first examine the  curve $\xi$ defined in \eqref{eq:curve-eta}. By \eqref{eq:A-diag} and \eqref{eq:dj7}, we have for  $2\leq k\leq n$, 
  \begin{equation} \label{eq:dj10}
     x_k^*(\lambda  x_1^*) = \lambda_k  x_k^*(x_1^*)
  \end{equation}
  and hence, by \eqref{eq:dj9}, there exist integers $p_2,\ldots,p_n \ge 2$ and reals $c_2,\ldots,c_n \ne 0$ such that
  \begin{equation*}
    c_k(\lambda  x_1^*)^{p_k} = \lambda_k c_k  (x_1^*)^{p_k} + o((x_1^*)^{p_k}).
  \end{equation*}
   This implies that $\lambda_k = \lambda^{p_k}$ and thus $x_k^*(\lambda  x_1^*) = \lambda^{p_k}  x_k^*(x_1^*)$. Taking $p_k$-th derivative on both sides gives $(x_k^*)^{(p_k)}(\lambda x_1^*) = (x_k^*)^{(p_k)}(x_1^*)$. Hence $(x_k^*)^{(p_k)}(\lambda^j x_1^*) =  (x_k^*)^{(p_k)}(x_1^*)$ for all $j \in \N$. Letting $j \to \infty$, we get $(x_k^*)^{(p_k)}(x_1^*) \equiv (x_k^*)^{(p_k)}(0) = c_kp_k!$. Combining this with \eqref{eq:dj9} yields
  \begin{equation*}
     x_k^*( x_1^*) = c_k  (x_1^*)^{p_k}.
  \end{equation*}
  Since the curve $\tilde\gamma$ is not contained in a hyperplane, we see that, for any non-zero vector $(b_1,\ldots,b_n)$, the sum $\sum_{k=1}^n b_k  x_k^*$ is not identically zero. Thus the integers $p_2,\ldots,p_n$ are mutually distinct. Hence the curve $\xi: (-\epsilon, \epsilon)\to \R^n$  is of the form $\xi(x_1^*)=(x_1^*, c_2 (x_1^*)^{p_2},\ldots, c_n (x_1^*)^{p_n})$. Without confusion, we simply write $\xi(t)=(t, c_2 t^{p_2},\ldots, c_n t^{p_n})$.

  We have now proved that, possibly after a permutation on coordinate axis, the curve $\tilde \gamma \colon (t_0-\delta_1,t_0+\delta_2) \to \R^n$ for some $\delta_1,\delta_2>0$, can be re-pararemtrized by
  \begin{equation*}
    t \mapsto (t,c_2t^{p_2},\ldots,c_nt^{p_n}),  \; t\in (-\epsilon,\epsilon)
  \end{equation*}
  for some integers $2 \le p_2 < p_3 < \cdots < p_n$ and reals $c_2,\ldots,c_n \ne 0$. Applying a further affine transformation $(u_1,u_2,\ldots,u_n) \mapsto (u_1, u_2/c_2,\ldots, u_n/c_n)$, we see that  $\gamma \colon (t_0-\delta_1,t_0+\delta_2) \to \R^n$, for some $\delta_1,\delta_2>0$, is an affine image of the curve $\eta$.  This completes the proof of Fact \ref{fact3}.
  \end{myproof}

  By Fact \ref{fact3}, it suffices to show that $p_k=k$ for all $k \in \{ 2,\ldots,n \}$. Observe that $\eta \colon (-\eps,\eps) \to \R^n$ given by Fact \ref{fact3} is an analytic simple curve which cannot be embedded in a hyperplane and it contains a non-trivial self-affine set, say $F$. Then there exists $t_1\in (-\eps,\eps)\setminus\{0\}$ such that $\eta(t_1)$ is the fixed point of a mapping  of the affine IFS defining $F$.   Therefore, applying the previous argument (Fact 3) once more (in which $\gamma$ is replaced by $\eta$), we find integers $2 \le q_2 < q_3 < \cdots < q_n$ and an interval $(t_1-\delta_1', t_1+\delta_2')\subset (-\epsilon,\epsilon)$ for some some $\delta_1',\delta_2'>0$ such that, under a suitable invertible linear transformation $B'$, the curve
  \begin{equation*}
    t \mapsto B'(\eta(t)-\eta(t_1))
  \end{equation*}
  defined on $(t_1-\delta_1',t_1+\delta_2')$  can be re-parametrized by
  \begin{equation*}
    t \mapsto (t,t^{q_2},\ldots,t^{q_n}).
  \end{equation*}
  This means that, writing $B' = (b_{kj})_{1 \le k,j \le n}$, we have
  \begin{equation} \label{eq:dj12}
    \sum_{j=1}^n b_{kj} (t^{p_j}-t_1^{p_j}) = \biggl( \sum_{j=1}^n b_{1j}(t^{p_j}-t_1^{p_j}) \biggr)^{q_k}
  \end{equation}
  for all $t \in (t_1-\delta_1',t_1+\delta_2')$ and $k \in \{ 2,\ldots,n \}$, where $p_1=1$.  By analyticity, \eqref{eq:dj12} holds for all $t \in \R$.

  We will next compare the degrees of polynomials of $t$ on both sides of \eqref{eq:dj12} for all $k \in \{ 2,\ldots,n \}$. Let $d = \deg(\sum_{j=1}^n b_{1j}(t^{p_j}-t_1^{p_j})) \in \{ 1,p_2,\ldots,p_n \}$. When $k$ runs over $\{ 2,\ldots,n \}$, the degrees of the right-hand side of \eqref{eq:dj12} are $dq_2,dq_3,\ldots,dq_n$, whereas the left-hand side has degree in $\{ 1,p_2,\ldots,p_n \}$. Therefore,
  \begin{equation*}
    \{ dq_2,dq_3,\ldots,dq_n \} \subset \{ 1,p_2,\ldots,p_n \}
  \end{equation*}
  which implies that
  \begin{equation} \label{eq:dj13}
    p_k = dq_k
  \end{equation}
  for all $k \in \{ 2,\ldots,n \}$. Since $d \in \{ 1,p_2,\ldots,p_n \}$, we must have $d=1$ (otherwise, by \eqref{eq:dj13}, $q_k=1$ for some $k \in \{ 2,\ldots,n \}$ which is a contradiction). But since $d=1$, we may write \eqref{eq:dj12} as
  \begin{equation*}
    \sum_{j=1}^n b_{kj}(t^{p_j}-t_1^{p_j}) = (c(t-t_1))^{p_k}
  \end{equation*}
  for all $k \in \{ 2,\ldots,n \}$. In particular, this shows that $(t-t_1)^{p_n}$ is a linear combination of $(t-t_1), (t^{p_2}-t_1^{p_2}), \ldots, (t^{p_n}-t_1^{p_n})$. Since $t_1 \ne 0$, all powers $t^j$, $j \in \{ 1,\ldots,p_n \}$, appear in $(t-t_1)^{p_n}$ with non-degenerate coefficients, and it follows that $p_k=k$ for all $k \in \{ 2,\ldots,n \}$.
\end{proof}

\begin{remark}
  (1) Bandt and Kravchenko showed that there are plenty of $C^1$ planar self-affine curves (i.e.\ self-affine sets that are $C^1$ planar curves); see \cite[Theorem 2]{BandtKravchenko2011}. Furthermore, in \cite[Theorem 3(ii)]{BandtKravchenko2011}, they showed that parabolic arcs and straight line segments are the only simple $C^2$ planar self-affine curves. This result also follows from Theorem \ref{thm:main} by a simple modification. It would be interesting to know that if a self-affine set $E$ is contained in a $C^2$ planar curve, then does there exists an analytic curve containing $E$?

  (2) The analyticity assumption in Theorem \ref{thm:main} is well motivated since for each $k \in \N$ it is easy to construct a non-parabolic $C^k$ planar curve containing a self-affine set. To see this, start from a  piece of parabolic curve and change a small part of it so that the new  curve is $C^k$. Clearly the obtained curve still contains a self-affine set. Due to this, it would be interesting to know if there exists a self-affine set $E$ which is a subset of a strictly convex $C^2$ planar curve, but is not a subset of any parabolic curve. Also, when can a self-affine set intersect an analytic curve in a set of positive measure for some relevant measure such as the self-affine measure? In the self-conformal case, this property implies that the whole set is contained in an analytic curve; see \cite[Theorem 2.1]{Kaenmaki2003}.
\end{remark}

\section{Self-affine sets and algebraic surfaces}

In this section, we prove Theorem \ref{thm:no_self-affine} and introduce self-affine polynomials.

\begin{proof}[Proof of Theorem \ref{thm:no_self-affine}]
  Let $P \colon \R^n \to \R$ be a non-constant polynomial with real coefficients such that $S(P)$ is compact. Suppose to the contrary that there exists a non-trivial self-affine set $E$ contained in $S(P)$. Let $f$ be one of the mappings of the affine IFS defining $E$ and set $P_j = P \circ f^{-j}$ for all $j \in \N$. Observe that the degree of $P_j$ is at most $\deg(P)$. It is easy to see that $S(P_j) = f^j(S(P))$ for all $j \in \N$ and therefore $\diam(S(P_j)) \to 0$ as $j \to \infty$. By the assumption, we have $f^j(E) \subset f^j(S(P)) = S(P_j)$ for all $j \in \N$, and by the invariance, we have $f^j(E) \subset f^{j-1}(E) \subset \cdots \subset E$ for all $j \in \N$.

  Since the ring of polynomials having degree at most $\deg(P)$ is finite dimensional there exist $P_{k_1},\ldots,P_{k_m}$ such that each $P_j$ is a linear combination of these polynomials. Choose $j$ so large that
  \begin{equation*}
    \diam(S(P_j)) < \min_{i \in \{ 1,\ldots,m \}} \diam(f^{k_i}(E)) = \diam\biggl(\bigcap_{i=1}^m f^{k_i}(E)\biggr).
  \end{equation*}
  But since $P_j = \sum_{i=1}^m c_iP_{k_i}$ for some $c_i$, we have
  \begin{equation*}
    \bigcap_{i=1}^m f^{k_i}(E) \subset \bigcap_{i=1}^m S(P_{k_i}) \subset S(P_j).
  \end{equation*}
  This contradiction finishes the proof.
\end{proof}

\begin{remark}
  By slightly modifying the above argument, we can prove the following stronger result:  {\it If $S(P)$ is an algebraic surface and there exists a contractive affine map $f$ such that $S(P)$ contains the fixed point $z$ of $f$ and a non-periodic orbit $\{ f^n(x) \}$ for some $x$, then $S(P)$ is unbounded}. To see this, choose $k_1<\ldots<k_m$ so that each $P_n$ is a linear combination of  the polynomials $P_{k_1},\ldots,P_{k_m}$. If $S(P)$ is bounded, then we can pick $j$ large enough so that ${\rm diam}(S(P_j))<|z-f^{k_m}(x)|$. This is a contradiction since $S(P_j)\supset \bigcap_{i=1}^m S(P_{k_i}) \supset \{z, f^{k_m}(x)\}$.
\end{remark}

\begin{example} \label{ex:paraboloid}
  It is clear that a hyperplane can contain a non-trivial self-affine set. In this example, we show that also other kinds of non-compact algebraic surfaces  can have this property.  Let $P \colon \R^n \to \R$, $P(x_1,\ldots,x_n) = x_1^2 + \cdots + x_{n-1}^2 - x_n$, and observe that, by Lemma \ref{thm:parabola-self-affine}, the parabola $\{ (x_1,\ldots,x_n) \in \R^n : x_n = x_1^2 \text{ and } x_2 = \cdots = x_{n-1} = 0 \} \subset S(P)$ contains non-trivial self-affine sets. It is also easy to see that $S(P)$ contains self-affine sets having dimension larger than one. Fix an interval $[a,b] \subset \R$ and define a mapping $\eta \colon [a,b]^{n-1} \to \R^n$ by setting $\eta(x_1,\ldots,x_{n-1}) = (x_1,\ldots,x_{n-1},x_1^2+\cdots+x_{n-1}^2)$. Let $\{ c_i(x_1,\ldots,x_{n-1}) + (d_i,\ldots,d_i) \}_{i=1}^\ell$ be an affine IFS on $\R^{n-1}$ so that $[a,b]^{n-1}$ is the self-affine set generated by it. Define $f_i \colon \R^n \to \R^n$ by setting
  \begin{equation*}
    f_i(x_1,\ldots,x_n) =
    \begin{pmatrix}
      c_i & 0 & \cdots & 0 & 0 \\
      0 & c_i & \cdots & 0 & 0 \\
      \vdots & \vdots & \ddots & \vdots & \vdots \\
      0 & 0 & \cdots & c_i & 0 \\
      2c_id_i & 2c_id_i & \cdots & 2c_id_i & c_i^2
    \end{pmatrix}
    \begin{pmatrix}
      x_1 \\ x_2 \\ \vdots \\ x_{n-1} \\ x_n
    \end{pmatrix}
    +
    \begin{pmatrix}
      d_i \\ d_i \\ \vdots \\ d_i \\ (d-1)d_i^2
    \end{pmatrix}
  \end{equation*}
  for all $(x_1,\ldots,x_n) \in \R^n$ and $i \in \{ 1,\ldots,\ell \}$. Since $f_i(\eta(x_1,\ldots,x_{n-1})) = \eta(c_ix_1+d_i,\ldots,c_ix_{d-1}+d_i)$ the image $\eta([a,b]^{n-1}) \subset S(P)$ is invariant under the affine IFS $\{ f_i \}_{i=1}^\ell$.
\end{example}

 We shall next introduce a general condition which guarantees the algebraic surface to contain self-affine sets. Suppose that $P \colon \R^n \to \R$ is a non-constant polynomial with real coefficients. We say that a contractive invertible affine map $f$ is a \emph{scaling factor} for $P$ if there exists a constant $C \in \R$ such that
\begin{equation} \label{eq:scaling-factor}
  P \circ f = CP.
\end{equation}
A polynomial $P$ is called \emph{self-affine} if it has two scaling factors with distinct fixed points.

\begin{example}
  Let $P \colon \R^2 \to \R$, $P(x_1,x_2) = x_2-x_1$. It is easy to see that $f \colon \R^2 \to \R^2$, $f(x_1,x_2) = \tfrac12(x_1,x_2)$, and $g \colon \R^2 \to \R^2$, $g(x_1,x_2) = \tfrac12(x_1+1,x_2+1)$, are scaling factors for $P$ and have distinct fixed points.
\end{example}

The following proposition shows that a polynomial $P$ being self-affine is sufficient for the inclusion of self-affine sets.

\begin{proposition} \label{thm:selfaffinepoly}
  If $P \colon \R^n \to \R$ is a self-affine polynomial, then $S(P)$ contains a non-trivial self-affine set.
\end{proposition}

\begin{proof}
  Let $f$ be a scaling factor for $P$ with a constant $C$. Note that there exists a non-singular $d \times d$ matrix $M$ with $\| M \| < 1$ and $a \in \R^n$ so that $f(x) = Mx+a$ for all $x \in \R^n$. Observe that
  \begin{equation*}
    f^j(x) = M^jx + \sum_{i=0}^{j-1} M^ia \to \sum_{i=0}^\infty M^ia =: x_0
  \end{equation*}
  as $j \to \infty$, where $x_0 \in \R^n$ is the fixed point of $f$. Choose $x \in \R^n$ such that
  \begin{equation*}
    |P(x_0)|+1 < |P(x)|.
  \end{equation*}
  Such a point $x$ exists since $P$ is not bounded. Since
  \begin{equation*}
    C^jP(x) = P \circ f^j(x) \to P(x_0)
  \end{equation*}
  as $j \to \infty$ we may choose $j$ large enough so that $|C^jP(x)| < |P(x_0)|+1$. Thus $|C|<1$.

  Let $h$ and $g$ be scaling factors for $P$ with distinct fixed points. If $f$ is any finite composition of the mappings $h$ and $g$, then $f$ is a scaling factor for $P$. If $C$ is the constant associated to the scaling factor $f$, then the above reasoning implies that $|C|<1$. Furthermore, if $x_0$ is the fixed point of $f$, then $P(x_0) = P \circ f(x_0) = CP(x_0)$. Since $|C|<1$, this implies $P(x_0)=0$ and $x_0 \in S(P)$. Recalling that $S(P)$ is closed it thus contains the self-affine set generated by the affine IFS $\{ h,g \}$.
\end{proof}

\begin{remark}
  It would be interesting to characterize all the algebraic surfaces associated to self-affine polynomials. For example, in the two-dimensional case, is the surface always contained in a line through the origin? Of course, the ultimate open question here is to characterize all the algebraic surfaces containing self-affine sets.
\end{remark}

\begin{ack}
  Feng was partially supported by the HKRGC GRF grants (projects CUHK401013, CUHK14302415). The authors  are grateful to Christoph Bandt for many valuable comments to improve the paper.
\end{ack}


\begin{thebibliography}{1}

\bibitem{BandtKravchenko2011}
C.~Bandt and A.~Kravchenko.
\newblock Differentiability of fractal curves.
\newblock {\em Nonlinearity}, 24(10):2717--2728, 2011.

\bibitem{Falconer1990}
K.~Falconer.
\newblock {\em Fractal geometry}.
\newblock John Wiley \& Sons Ltd., Chichester, 1990.
\newblock Mathematical foundations and applications.

\bibitem{Hutchinson1981}
J.~E. Hutchinson.
\newblock Fractals and self-similarity.
\newblock {\em Indiana Univ. Math. J.}, 30(5):713--747, 1981.

\bibitem{Kaenmaki2003}
A.~K{\"a}enm{\"a}ki.
\newblock On the geometric structure of the limit set of conformal iterated
  function systems.
\newblock {\em Publ. Mat.}, 47(1):133--141, 2003.

\bibitem{Kaenmaki2006}
A.~K{\"a}enm{\"a}ki.
\newblock Geometric rigidity of a class of fractal sets.
\newblock {\em Math. Nachr.}, 279(1):179--187, 2006.

\bibitem{Mattila1982}
P.~Mattila.
\newblock On the structure of self-similar fractals.
\newblock {\em Ann. Acad. Sci. Fenn. Ser. A I Math.}, 7(2):189--195, 1982.

\bibitem{MayerUrbanski2003}
V.~Mayer and M.~Urba{\'n}ski.
\newblock Finer geometric rigidity of limit sets of conformal {IFS}.
\newblock {\em Proc. Amer. Math. Soc.}, 131(12):3695--3702 (electronic), 2003.

\end{thebibliography}
\end{document}